\documentclass[preprint,12pt]{elsarticle}

\usepackage{amssymb}
\usepackage{eufrak}
\usepackage{calrsfs}
\usepackage{mathrsfs}
\usepackage{amsmath}
\usepackage{amsthm}
\newtheorem{theorem}{Theorem}

\newtheorem{definition}[theorem]{Definition}

\newtheorem{lemma}[theorem]{Lemma}

\newtheorem{proposition}[theorem]{Proposition}

\begin{document}

\begin{frontmatter}


\author[label1,label2]{Sh. Haghi}
\ead{sh.haghi@ipm.ir}
\author[label1,label2]{H. R. Maimani}
 \ead{maimani@ipm.ir}
 \author[label1,label2]{A. Seify}
 \ead{abbas.seify@gmail.com}
\address[label1]{Mathematics Section, Department of Basic Sciences, Shahid Rajaee Teacher Training University, P. O. BOX 16783-163,Tehran, Iran}
\address[label2]{School of Mathematics, Institute for Research in Fundamental Sciences (IPM), P. O. BOX 19395-5746, Tehran, Iran}

\title{Star-critical Ramsey number of $K_4$ versus $F_n$}




\begin{abstract}
For two graphs $G$ and $H$, the Ramsey number $r(G,H)$ is the smallest positive integer $r$, such that any red/blue coloring of the edges of graph $K_r$ contains either a red subgraph that is isomorphic to $G$ or a blue subgraph that is isomorphic to $H$. Let $S_k=K_{1,k}$ be a star of order $k+1$ and $K_n\sqcup S_k$ be a graph obtained from $K_n$ by adding a new vertex $v$ and joining $v$ to $k$ vertices of $K_n$. The star-critical Ramsey number $r_*(G, H)$ is the smallest positive integer
$k$ such that any red/blue coloring of the edges of graph $K_{r-1}\sqcup S_k$ contains either a red subgraph that is isomorphic to $G$ or a blue subgraph that is isomorphic to $H$, where $r=r(G,H)$.
In this paper, it is shown that $r_*(F_n,K_4)=4n+2$, where $n\geq{4}$.
\end{abstract}


\end{frontmatter}


\section{Introduction and Background}
Let $G=(V(G), E(G))$ denote a finite simple graph on the vertex set $V(G)$ and the edge set $E(G)$. The \textit{order} of a graph $G$ is $|V(G)|$. The subgraph of $G$ \textit{induced} by $S\subseteq{V(G)}$, $G[S]$, is a graph with vertex set $S$ and two vertices of $S$ are adjacent in $G[S]$ if and only if they are adjacent in $G$. For a vertex $v \in V(G)$, we denote the set of all neighbors of $v$ by $N(v)$. For the subset $S\subseteq{V(G)}$, $N(S)=\bigcup_{s\in{S}}N(s)$. 
The degree of a vertex $v$ in $G$ is denoted by $d_G(v)$ (for abbreviation $d(v)$). 
The graph $G-H$ is the subgraph of $G$ obtaining from the deletion
of the vertices of $H$ where $H$ is a subgraph of $G$.
\\
The cycle of order $n$ is denoted by $C_n$. We refer to a cycle of odd order as an odd cycle. We denote the complete graph on $r$ vertices by $K_r$. By \textit{triangle}, we refer to the complete graph $K_3$. A \textit{clique} is a subset of vertices of a graph, such that its induced subgraph is complete. The number of vertices of a largest clique in a graph $G$ is denoted by $\omega(G)$. The \textit{fan} graph, $F_n$, can be constructed by joining $n$ copies of the complete graph $K_3$ with a common vertex. We refer to this common vertex as the \textit{center} vertex of the fan graph. For any graph $G$ and positive integer $n$, the disjoint union of $n$ copies of $G$ is denoted by $nG$.
\\
For a positive integer $k$, a complete $k$-partite graph is a graph that can be partitioned into $k$ disjoint sets, such that no two vertices within the same set are adjacent, but every pair of vertices from two different sets are adjacent. Now, let $G$ be a complete $k$-partite graph with disjoint sets $A_1,\ldots, A_k$ which we sometimes write $G=(A_1,\ldots, A_k)$, then the graph $G$ is denoted by $K_{|A_1|,\ldots, |A_k|}$.
\\
For a red/blue edge-coloring of a graph $G$, the subgraph induced by red edges is denoted by $G^R$ and the subgraph induced by blue edges is denoted by $G^B$. we denote the set of all neighbors of $v$ in $G^R$ by $N^R(v)$ and in $G^B$ by $N^B(v)$. We sometimes refer to a vertex $u\in N^B(v)$ (or $u\in N^R(v)$) as a blue neighbor (or red neighbor) of $v$. The number of all blue neighbors of $v$ in $G$ is denoted by $d^B(v)$. In the other words, $d^B(v)=|N^B(v)|$. We define $d^R(v)$ similarly. 
 Let $H$ be a subgraph of $G$ or a subset of $V(G)$ and $v\in{V(G)}$. The set $N_H(v)$ is the set of neighbors of $v$ in $H$. In the other words, $N_H(v)=N(v)\cap H$. Also, we define the sets $N_H^B(v)=N^B(v)\cap H$ and $N_H^R(v)=N^R(v)\cap H$. The $d_H^B(v)$ and $d_H^R(v)$ are defined similarly.
\\
Let $S_k=K_{1,k}$ be a \textit{star} of order $k+1$ and $K_n\sqcup S_k$ be a graph obtained from $K_n$ by adding a new vertex $v$ and joining $v$ to $k$ vertices of $K_n$.
\\
Let $G$ and $H$ be two graphs. The \textit{Ramsey number} of $G$ and $H$ is the smallest positive integer $r$ such that every red/blue coloring of $K_r$ contains a red $G$ or a blue $H$. Note that for $r=r(G,H)$, there exists a critical red/blue edge-coloring graph $K_{r-1}$ which contains no monochromatic subgraph isomorphic to $G$ or $H$. We call such a red/blue edge-coloring a $(G,H)$-free coloring and the complete graph, $K_{r-1}$, with a $(G,H)$-free coloring as a $(G,H)$-free graph.
\\
For two graphs $G$ and $H$, the \textit{star-critical Ramsey number}, $r_*(G,H)$, is defined to be the smallest
integer $k$ such that every red/blue edge-coloring of $K_{r-1}\sqcup S_k$ contains a red $G$ or a blue $H$, where $r=r(G,H)$.
\\
The star-critical Ramsey number was defined by Hook and Isaak in \cite{H}. The following theorem was Shown by Zhen Li and Yusheng Li in \cite{Z}.

\begin{theorem}
Let $n\geq{2}$ be an integer. Then $r_*(F_n, K_3)=2n+2$.
\end{theorem}

In this paper, we prove the following theorem.

\begin{theorem}\label{t1}
Let $n\geq{4}$ be an integer. Then $r_*(F_n, K_4)=4n+2$.
\end{theorem}

\section{Proof of Theorem \ref{t1}}
In order to prove our theorem, we need some theorems and lemmas.

\begin{lemma}\cite{P}\label{l1}
Let $m\geq{1}$ and $n\geq{2}$ be integers. Then $r(mK_2,K_n)=n+2m-2$. In particular, $r(nK_2,K_4)=2n+2$.
\end{lemma} 

\begin{theorem}\cite{G}\label{t2}
$r(F_n,K_3) = 4n+1$ for $n\geq{2}$.
\end{theorem}

\begin{theorem}\cite{S}\label{t3}
Let $n\geq{2}$ be an integer. Then $r(F_n, K_4)=6n+1$.
\end{theorem}

\begin{proposition}\label{p1}
Let $G=K_{6n}$ be a $(F_n,K_4)$-free graph, where $n\geq{2}$. For each $v\in{V(G)}$, $2n-1\leq{d^R(v)}\leq{2n+1}$. Consequently, $4n-2\leq{d^B(v)}\leq{4n}$.
\end{proposition}

\begin{proof}{
If $d^R(v)\geq{2n+2}$, then by Lemma \ref{l1}, $G[N^R(v)]$ contains a red $nK_2$ or a blue $K_4$. If $d^R(v)\leq{2n-2}$, then $d^B(v)\geq{4n+1}$. By Theorem \ref{t2}, $G[N^B(v)]$ contains a blue $K_3$ or a red $F_n$. In both cases, we reach to a contradiction.
}\end{proof}

\begin{proposition}\label{k3}
Let $n \geq 2$ and $G$ be a $(F_n, K_4)$-free graph of order $6n$ which contains a red $K_{2n}$, say $K$. Then $G-K$ contains no blue triangle.
\end{proposition}

\begin{proof}{
By contrary, suppose that $G-K$ contains a blue triangle, say $T=xyz$. If $d^{R}_{K}(u) \geq 2$, for some $u \in \{x,y,z\}$, then $K\cup\{u\}$ contains a red $F_n$, which is a contradiction. Therefore, $d^{B}_{K}(u) \geq 2n-1$, for every $u \in V(T)$. Thus, $N^{B}(x) \cap N^{B}(y) \cap N^{B}(z) \neq \emptyset$, which implies that $G$ contains a blue $K_4$, a contradiction.
}\end{proof}

\begin{proposition}\label{c5c7}
Let $n \geq {4}$ and $G$ be a $(F_n, K_4)$-free graph of order $6n$ which contains a red $K_{2n}$, say $K$. Also, suppose that $C_t$ is a shortest blue odd cycle in $G-K$. Then $t=5$ or $t=7$.
\end{proposition}

\begin{proof}{
Proposition \ref{k3} implies that $t \geq 5$. Let $u$ be an arbitrary vertex of $C_t$. By Proposition \ref{p1}, since $d^B(u)\geq{4n-2}$, we conclude that $d_{G-(K\cup C_t)}^B(u) \geq 2n-4$. Since $G-K$ contains no blue triangle, each two adjacent vertices of $C_t$, say $u$ and $v$, have at least $4n-8$ distinct blue neighbors in $G-(K\cup C_t)$. Thus, $4n-t \geq 4n-8$, which implies that $t\leq{8}$ and hence $t=5$ or $t=7$.
}\end{proof}

\begin{proposition}\label{c5}
Let $n \geq {4}$ and $G$ be a $(F_n, K_4)$-free graph of order $6n$ which contains a red $K_{2n}$, say $K$. Then $G-K$ contains no blue $C_5$.
\end{proposition}

\begin{proof}{
By contrary, suppose that $G$ contains a blue $C_5$. Note that Proposition \ref{k3} implies that this cycle is an induced cycle in $G^B$. Let $V(C_5)=\{u_1, \ldots, u_5\}$ such that $u_i$ is adjacent to $u_{i+1}$, for $i=1, \ldots, 4$ and $u_5$ is adjacent to $u_1$. For each $u_i \in V(C_5)$, we have $|N_{G-(K\cup C_5)}^B(u_i)| \geq 2n-4$. 
\\
Let $X_i \subseteq N_{G-(K\cup C_5)}^B(u_i)$ with $|X_i|=2n-4$, for $i=1,2, \ldots, 5$. Then Proposition \ref{k3} implies that $R_1=X_1 \cup \{u_2, u_5\}$ and $R_2=X_2 \cup \{u_1, u_3\}$ are disjoint red cliques of order $2n-2$ (See Figure \ref{f1}). Let $a, b$ and $c$ be the remaining vertices of $G$. Proposition \ref{k3} implies that at least one of the edges between $a, b, c$ is red, for instance $ab$.
\\
Suppose that all edges between two sets $\{a, b, c\}$ and  $\{u_1, u_2\}$ are red. Then all edges between $c$ and $R_1\cup R_2-\{u_1, u_2\}$ are blue, otherwise $R_1\cup\{a,b,c\}$ or $R_2\cup\{a,b,c\}$ contains a red $F_n$ with the center $u_2$ or $u_1$, respectively. So, $d_{G-K}^B(c)\geq 4n-6=(2n-2)+(2n-2)-2$ and consequently $d^{B}_{K}(c) \leq 6$, since $d^B(c)\leq{4n}$. Now, since $n \geq 4$, $K\cup\{c\}$ contains a red $F_n$, which is a contradiction.
\\ 
So, with no loss of generality, we may assume that $a$ has a blue edge to $u_1$. Proposition \ref{k3} implies that all edges between $a$ and $R_1$ are red. Let $R_1^\prime=R_1\cup \{a\}$. By Proposition \ref{k3}, at least one of the edges between $b, c, u_4$ is red, for instance $bc$. Similarly, all edges of the set $\{b,c\}$ to the set $\{u_1,u_2\}$ can not be red.  If all edges of $\{b,c\}$ to the set $\{u_1,u_2\}$ are red, then all edges between $u_4$ and $R_1\cup R_2-\{u_1, u_2\}$ are blue, otherwise $R_1\cup\{b,c,u_4\}$ or $R_2\cup\{b,c,u_4\}$ contains a red $F_n$ with the center $u_2$ or $u_1$, respectively. It means $d_{G-K}^B(u_4) \geq 4n-6$ and consequently, $d_{K}^{B}(u_4) \leq 6$. Since $n \geq 4$, $K \cup \{u_4\}$ contains a red $F_n$, which is a contradiction. 
\\
With no loss of generality, suppose that $bu_1$ or $bu_2$ is blue (Note that this case for $c$ is similar).
\\ 
If $bu_1$ is blue, then by Proposition \ref{k3} all edges between $b$ and $R_1^\prime$ are red. So, $R_1^{\prime\prime}=R_1^\prime \cup \{b\}$ is a red clique of order $2n$. Since $G$ contains no red $F_n$, $d_{R_{1}^{\prime\prime}}^{R}(u_3), d_{R_{1}^{\prime\prime}}^{R}(u_4) \leq 1$. So, $u_3$ and $u_4$ have at least $2n-2$ blue common neighbors in $R_1^{\prime\prime}$, a contradiction with the fact that $G-K$ has no blue triangle.
\\
If $bu_2$ is blue, then all the edges between $b$ and $R_2$ are red, by Proposition \ref{k3}. Then $R_2^\prime=R_2\cup \{b\}$ is a red clique of order $2n-1$. By Proposition \ref{k3}, one of the edges $cu_1$ or $cu_2$ is red, for instance $cu_1$. The edge $cu_4$ is blue, otherwise $R_2^\prime\cup \{c,u_4\}$ contains a red $F_n$ with the center $u_1$. Thus, $cu_3$ and $cu_5$ are red since $G-K$ contains no blue $K_3$. Therefore $R_2^\prime \cup \{c, u_5\}$ contains a red $F_n$ with the center $u_3$, a contradiction.
}\end{proof}

\begin{proposition}\label{c7}
Let $n \geq {4}$ and $G$ be a $(F_n, K_4)$-free graph of order $6n$ which contains a red $K_{2n}$, say $K$. Then $G-K$ contains no blue $C_7$.
\end{proposition}

\begin{proof}{
By contrary, suppose that $G$ contains a blue $C_7$. Propositions \ref{k3} and \ref{c5} yield that this cycle is an induced cycle in $G^B$. Let
$V(C_7)=\{u_1, \ldots, u_7\}$ such that $u_i$ is adjacent to $u_{i+1}$, for $i=1, \ldots, 6$ and $u_7$ is adjacent to $u_1$. By Proposition \ref{p1}, for every $u_i \in V(C_7)$, we have $d_{G-(K\cup C_7)}^B(u_i) \geq 2n-4$. Let $X_i \subseteq N_{G-(K\cup C_7)}^{B}(u_i)$, for $i=1,2, \ldots, 7$.
Then $R_1= X_1 \cup \{u_2, u_7\}$ and $R_2 = X_2 \cup \{u_1, u_3\}$ are red cliques of order $2n-2$ since $G-K$ contains no blue triangle (See Figure \ref{f1}).
It is easy to see that $R_2\cup \{u_4,u_5,u_6\}$ induces a red $F_n$ with the center  $u_1$, a contradiction.
}\end{proof}

Now, we can prove the following.

\begin{lemma}\label{l2}
Let $n \geq {4}$ be an integer and $G^R$ be the graph induced by red edges from a $(F_n,K_4)$-free coloring of $G=K_{6n}$. If $G^R$  contains $K_{2n}$, then it contains $3K_{2n}$.
\end{lemma}

\begin{figure}
\centering \includegraphics[width=150mm]{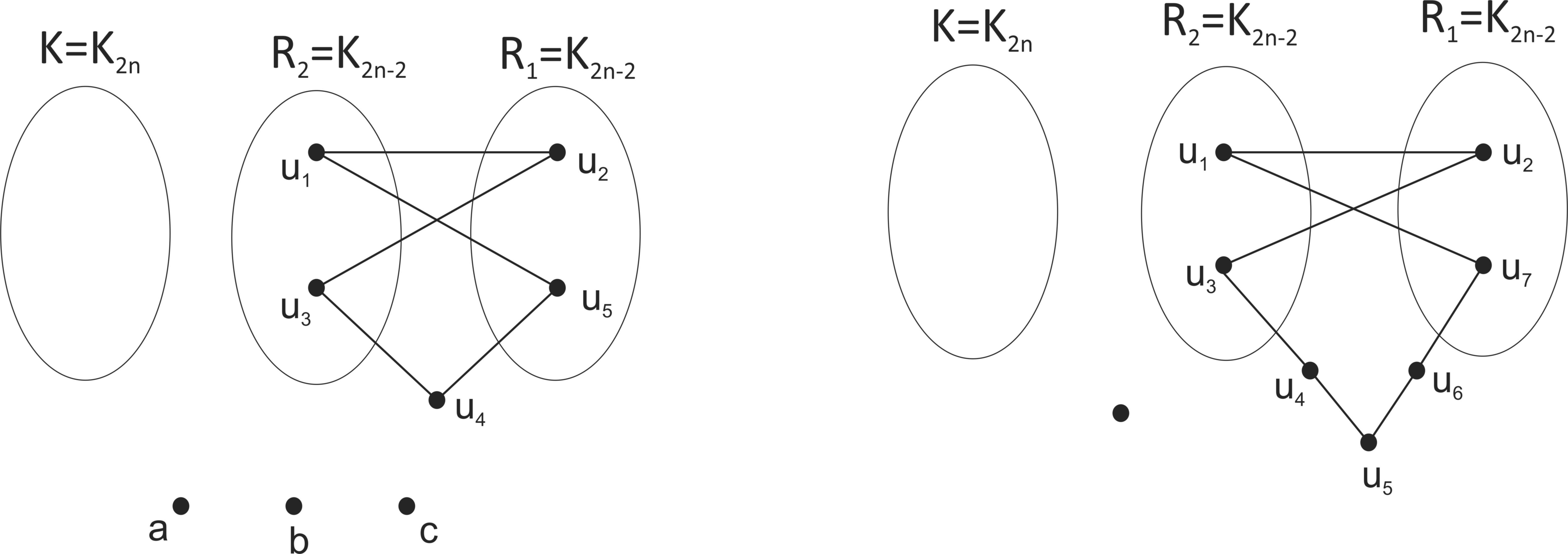}
\\
\caption{$G-K$ containing an induced blue $C_5$ or $C_7$. }\label{f1}
\end{figure}

\begin{proof}{
By Propositions \ref{k3}, \ref{c5} and \ref{c7}, we can conclude that $G-K$ contains no blue odd cycle. Thus,  $G^B[G-K]$ is bipartite. Since $G$ contains no red $F_n$, the size of each part is at most $2n$ and hence there exists a $3K_{2n}$ in $G^R$.  
}\end{proof}

\begin{figure} 
\centering \includegraphics[width=50mm]{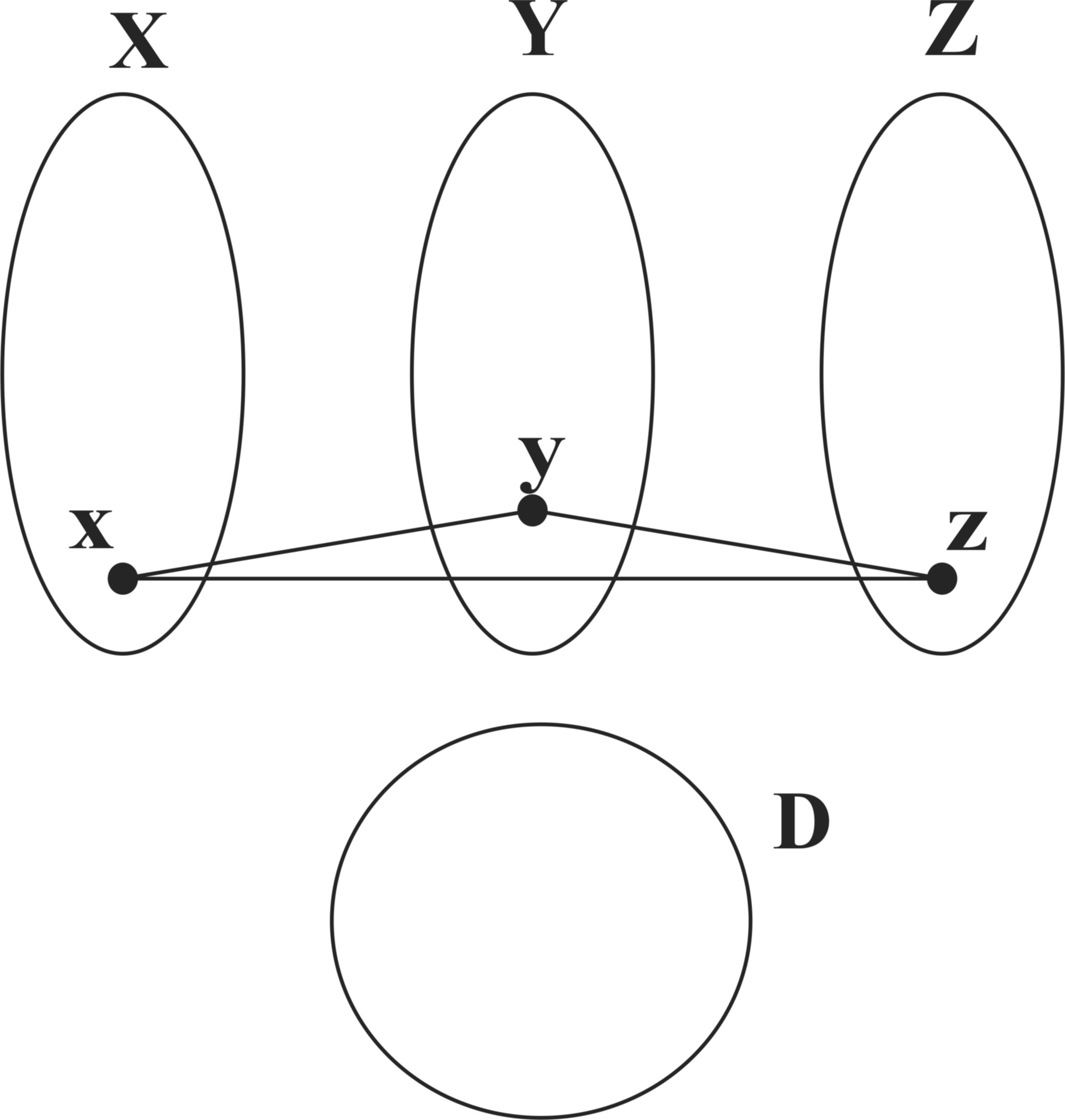}\\
\caption{The graph $G=K_{6n}$ containing a blue triangle  $T$.}\label{f2}
\end{figure}

\begin{lemma}\label{l4}
Suppose $G=K_{6n}$ is a $(F_n,K_4)$-free graph, where $n\geq{4}$. Then $\omega(G^R)=2n$.
\end{lemma}

\begin{proof}{
Let $G=K_{6n}$ be a $(F_n,K_4)$-free graph such that $\omega(G^R)\leq{2n-1}$. Since $R(F_n,K_3)=4n+1$, $G$ has a blue $K_3$, say $T=xyz$. Since $d_{G^B}(s)\geq{4n-2}$ , for all $s\in{V(G)}$, we conclude that $d_{G-T}^{B}(v) \geq 4n-4$, for all $v \in V(T)$.
Note that $|N^B(u) \cap N^B(v)| \geq 2n-5=(4n-4)+(4n-4)-(6n-3)$, for $u,v \in V(T)$.
\\
 Let $N_{xy}=N^B(x) \cap N^B(y)$. Define $N_{xz}$ and $N_{yz}$ similarly. If there exists a blue edge in $N_{xy}$, then $G$ contains a blue $K_4$, a contradiction. Similar argument holds for $N_{yz}$ and $N_{xz}$. Therefore $N_{xy}$, $N_{xz}$ and $N_{yz}$ are red cliques of order at least $2n-5$. Note that $N_{xy}\cap N_{xz}=N_{xy}\cap N_{yz}=N_{xz} \cap N_{yz}=\emptyset$ and all the edges between $z$ and  $N_{xy}$ are red, otherwise $G$ contains a blue $K_4$. So, $Z=N_{xy}\cup\{z\}$ is a red clique of order at least $2n-4$. Thus, $\omega(G^R)\geq{2n-4}$. Similarly, $X=N_{yz}\cup\{x\}$ and $Y=N_{xz}\cup\{y\}$ are the red cliques of order at least $2n-4$. Let $X^\prime \subseteq X$ be a red clique of order $2n-4$ which contains $x$. Define $Y^\prime$ and $Z^\prime$ similarly (See Figure \ref{f2}). 
\\ 
Let $D=V(G)-(X^\prime \cup Y^\prime \cup Z^\prime)$, then $|D| = 12$. Note that all edges between $x$ and $Y' \cup Z'$ are blue and consequently $d_{D}^{B}(x) \geq 6=4n-2-(2n-4+2n-4)$. Similar argument holds for $y$ and $z$. Since $|D|=12$, there are two vertices in $V(T)$ which have at least two blue common neighbors in $D$. With no loss of generality, assume that $|N_{D}^B(x)\cap N_{D}^B(y)|\geq{2}$. One can easily see that all edges  between $N_{D}^B(x)\cap N_{D}^B(y)$ and $Z$ are red since $G$ contains no blue $K_4$.  Therefore $\omega(G^R)\geq{2n-2}$. So, we consider two cases:
\vspace{0.6em}
\begin{figure} 
\centering \includegraphics[width=70mm]{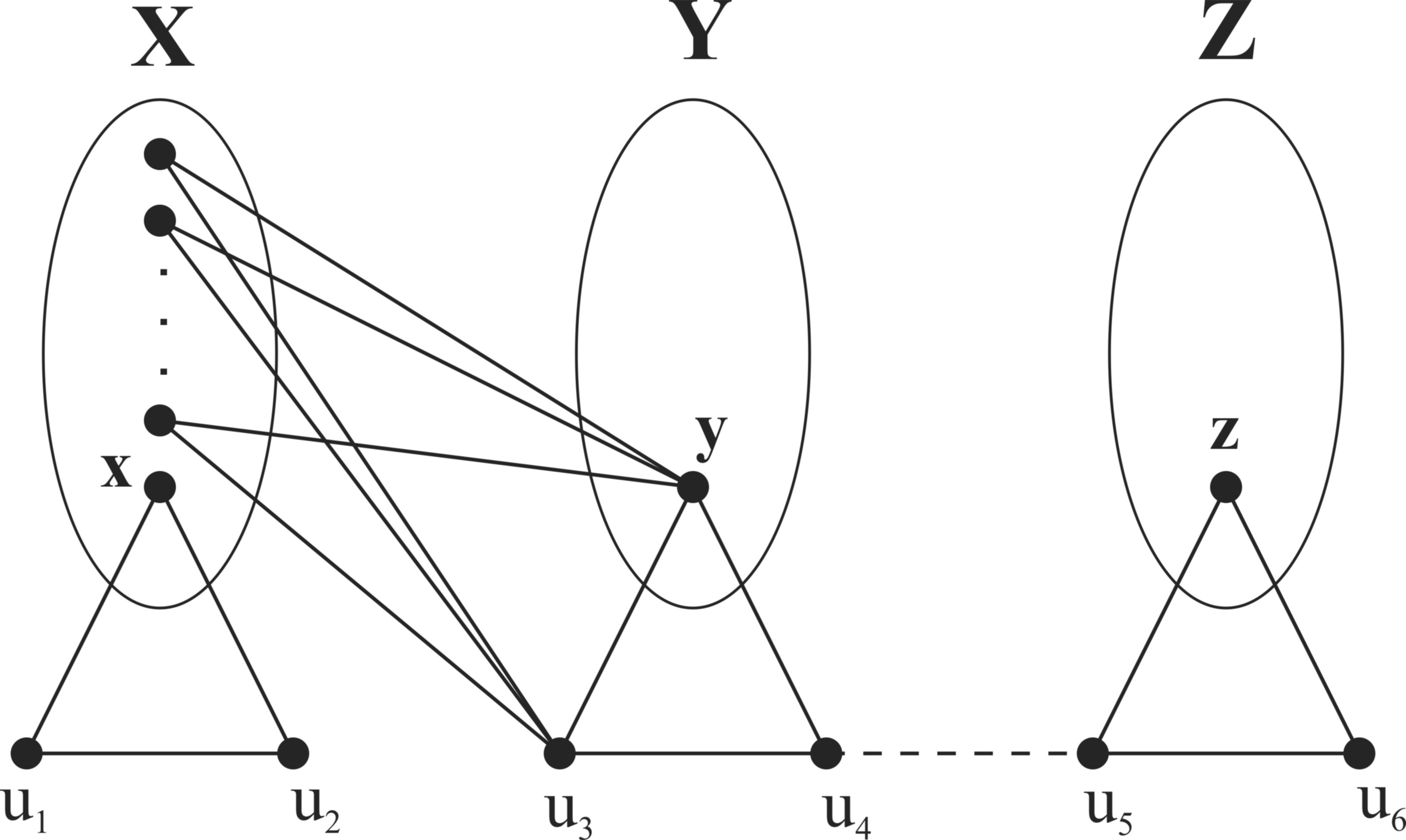}\\
\small{Figure 3}
\end{figure}
\\
\textbf{Case 1:} $\omega(G^R)=2n-2$.
\\
Since $\omega(G^R)=2n-2$ and $|N_{D}^B(s)|\geq{6}$, for all $s\in V(T)$, we conclude that $|N_{D}^B(x)\cap N_{D}^B(y)|=|N_{D}^B(x)\cap N_{D}^B(z)|=|N_{D}^B(y)\cap N_{D}^B(z)|=2$.
\\
Notice that $|X|=|Y|=|Z|=2n-2$. Let $D^\prime\subseteq{D}$ be the remaining $6$ vertices which each one of them is not the common neighbor of any two vertices of $x,y$ and $z$ in $G^B$. Let $D^\prime = \{u_1, \ldots, u_6\}$ such that $N^{B}_{D^\prime}(x)=\{u_1, u_2\}$, $N^{B}_{D^\prime}(y)=\{u_3, u_4\}$ and $N^{B}_{D^\prime}(z)=\{u_5, u_6\}$. 
\\
Assume that $u_1u_2$ is red. Then all edges between two sets $\{u_5, u_6\}$ and $Y-\{y\}$ are blue. Otherwise if there exists a red edge , for instance $u_5v$, where $v\in Y-\{y\}$, then $Y \cup \{u_1, u_2, u_5\}$ induces a red $F_n$ with the center $y$. Now, $\{u_5, u_6, v, z\}$ induces a blue $K_4$, where $v \in Y-\{y\}$. 
\\
So, we may assume that $u_1u_2$, $u_3u_4$ and $u_5u_6$ are blue. Note that since $G$ contains no blue $K_4$, there exists a red edge in $G[\{u_3, u_4, u_5, u_6\}]$, say $u_4u_5$. All edges between $u_3$ and $X-\{x\}$ are blue. Otherwise, if there exists a red edge $u_3v$, where $v \in X- \{x\}$, then $X \cup \{u_3, u_4, u_5\}$ induces a red $F_n$ with the center $x$, a contradiction (See Figure 3). 
\\
Now, all edges between $u_4$ and $X$ are red. Otherwise, if there exists a blue edge, for instance $u_4v$, where $v\in X-\{x\}$, then $\{u_3, u_4, v, y\}$ induces a blue $K_4$, a contradiction. But now $X \cup \{u_4\}$ induces a red clique of order $2n-1$, a contradiction.
\vspace{0.6em}
\\
\textbf{Case 2:} $\omega(G^R)=2n-1$.
\\
In this case, $|N_{D}^B(u)\cap N_{D}^B(v)|=3$, for some 2-subset $\{u,v\} \subseteq V(T)$. We divide this case into some subcases:
\vspace*{0.4em}
\\
\textbf{Subcase 2.1:} $|N_{D}^B(u)\cap N_{D}^B(v)|=3$, for every 2-subset $\{u,v\} \subseteq V(T)$.
\\
In this subcase we have $|X|=|Y|=|Z|=2n-1$. Let $D^\prime\subseteq{D}$ be the remaining $3$ vertices which each one of them is not the common neighbor of any two vertices of $x,y$ and $z$ in $G^B$. The vertex $v\in{D^\prime}$ has at least one blue edge to $X$, $Y$ and $Z$ since $\omega(G^R)=2n-1$. Suppose that $uv$ is a blue edge such that $u\in{X}$. 
Notice that $d^{B}_{Y\cup Z}(v) \geq 2n-3=(4n-2)-2-(2n-1)$ since $d_{G^B}(v)\geq{4n-2}$. Let $L_v=N_{Y\cup Z}^B(v)$, then $|L_v|\geq{2n-3}$. Since $G$ contains no red $F_n$, we conclude $d_{Y}^{R}(u) , d_{Z}^{R}(u) \leq 1$. Thus, $d_{L_v}^{B}(u) \geq 2n-5$. Let $L_u=N_{Y\cup Z}^B(u)\cap L_v$. Then $|L_u|\geq{2n-5}$ and $L_u$ induces a red clique, otherwise $\{u,v\}\cup L_u$ induces a blue $K_4$. Now, $Y\cup L_u$ or $Z\cup L_u$ induces a red $F_n$ since $n\geq{4}$, a contradiction.
\vspace{0.4em}
\\ 
\textbf{Subcase 2.2:} 
\\
$|N_{D}^B(x)\cap N_{D}^B(y)|=3$, $|N_{D}^B(x)\cap N_{D}^B(z)|=2$, $|N_{D}^B(y)\cap N_{D}^B(z)|=1$.
\\
 In this subcase, we have $|X|=2n-3$, $|Y|=2n-2$, $|Z|=2n-1$. Let $D^\prime\subseteq{D}$ be the remaining $6$ vertices which each one of them is not the common neighbor of any two vertices of $x,y$ and $z$ in $G^B$.  Also, let $D^\prime=\{u_1, u_2, \ldots, u_6\}$. We have $d_{D^\prime}^B(x)=1$, $d_{D^\prime}^B(y)=2$ and $d_{D^\prime}^B(z)=3$ since $d_{G^B}(s)\geq{4n-2}$, for all $s\in{V(G)}$. With no loss of generality, assume that $N_{D^\prime}^B(x)=\{u_1\}$, $N_{D^\prime}^B(y)=\{u_2,u_3\}$ and $N_{D^\prime}^B(z)=\{u_4,u_5,u_6\}$. One of the edges $u_4u_5$, $u_5u_6$ and $u_4u_6$ is red, otherwise $\{z,u_4,u_5,u_6\}$ induces a blue $K_4$. With no loss of generality, suppose that $u_5u_6$ is red. 
\\ 
Note that $\{u_1, u_2, u_3\}$ induces a blue $K_3$, otherwise there exists a red $F_n$ with center $z$. Also, $u_iu_4$ is blue, for $i=1,2,3$. If $u_1u_4$ is red, then $Y \cup \{u_1,u_4,u_5,u_6\}$ induces a red $F_n$, a contradiction. Also, if $u_iu_4$ is red, for $i=2$ or $i=3$, then $\{u_i, u_4, u_5, u_6\}$ induces a red $F_n$, a contradiction. Now, $\{u_1, u_2, u_3, u_4\}$ induces a blue $K_4$ which is a contradiction.  
\vspace{0.4em} 
\\
\textbf{Subcase 2.3:}
\\
$|N_{D}^B(x)\cap N_{D}^B(y)|=|N_{D}^B(x)\cap N_{D}^B(z)|=3$, $|N_{D}^B(y)\cap N_{D}^B(z)|=0$.
\\
In this subcase, we have $|X|=2n-4$, $|Y|=2n-1$, $|Z|=2n-1$. Let $D^\prime\subseteq{D}$ be the remaining $6$ vertices which each one of them is not the common neighbor of any two vertices of $x,y$ and $z$ in $G^B$. Also, let $D^\prime=\{u_1, u_2, \ldots, u_6\}$. We have $d_{D^\prime}^B(x)=0$, $d_{D^\prime}^B(y)=3$ and $d_{D^\prime}^B(z)=3$ since $d_{G^B}(s)\geq{4n-2}$, for all $s\in{V(G)}$. With no loss of generality, suppose that $N_{D^\prime}^B(y)=\{u_1,u_2,u_3\}$.  Note that there exists a red edge $u_iu_j$, where $1 \leq i<j \leq 3$, otherwise $\{y, u_1, u_2, u_3\}$ induces a blue $K_4$. Now, $Z \cup \{u_i, u_j\}$ induces a red $F_n$ with the center $z$, a contradiction.
\vspace{0.4em}
\\
\textbf{Subcase 2.4:}
\\
$|N_{D}^B(x)\cap N_{D}^B(y)|=|N_{D}^B(x)\cap N_{D}^B(z)|=3$, $|N_{D}^B(y)\cap N_{D}^B(z)|=1$.
\\
In this subcase, we have $|X|=2n-3$, $|Y|=2n-1$, $|Z|=2n-1$. Let $D^\prime\subseteq{D}$ be the remaining $5$ vertices which each one of them is not the common neighbor of any two vertices of $x,y$ and $z$ in $G^B$. Also, let $D^\prime=\{u_1, u_2, \ldots, u_5\}$. We have $d_{D^\prime}^B(y) \geq 2$ and $d_{D^\prime}^B(z) \geq 2$ since $d_{G^B}(s)\geq{4n-2}$, for all $s\in{V(G)}$. With no loss of generality, suppose that $\{u_1,u_2\}\subseteq N_{D^\prime}^B(y)$ and $\{u_3,u_4\}\subseteq N_{D^\prime}^B(z)$. 
\\
Suppose that $u_5 \in N^{B}(z)$. Then $G[\{u_3, u_4, u_5\}]$ has a red edge, say $u_4u_5$, otherwise $\{z, u_3, u_4, u_5\}$ induces a blue $K_4$. Now, $Y \cup \{u_4, u_5\}$ induces a red $F_n$, a contradiction. Similar argument holds for $y$.
\\
So, we may assume that $u_5 \in N^{R}(y) \cap N^{R}(z)$. Note that $u_4u_5$ is blue, otherwise $Y \cup \{u_4, u_5\}$ induces a red $F_n$ with the center $y$. Similarly, $u_iu_5$ is blue, for $i=1,2,3$.
\\
$u_1u_2$ is blue, otherwise $Z \cup \{u_1,u_2\}$ induces a red $F_n$ with the center $z$. Similarly, $u_3u_4$ is blue. Since $G$ contains no blue $K_4$, we may assume that $u_2u_3$ is red. Then $u_1u_4$ is blue, otherwise $X \cup \{u_1, u_2, u_3, u_4\}$ induces a red $F_n$. Also, note that $u_1u_3$ or $u_2u_4$ is blue, otherwise $X \cup \{u_1, u_2, u_3, u_4\}$ induces a red $F_n$. Suppose that $u_1u_3$ is blue. Now, $\{u_1, u_3, u_4, u_5\}$ induces a blue $K_4$, a contradiction.
\vspace{0.4em}
\\
\textbf{Subcase 2.5:}
\\
$|N_D^B(x)\cap N_D^B(y)|=|N_D^B(x)\cap N_D^B(z)|=3$, $|N_D^B(y)\cap N_D^B(z)|=2$.
\\
In this subcase, we have $|X|=2n-2$, $|Y|=2n-1$, $|Z|=2n-1$. Let $D^\prime\subseteq{D}$ be the remaining $4$ vertices which each one of them is not the common neighbor of any two vertices of $x,y$ and $z$ in $G^B$. Also, let $D^\prime=\{u_1, u_2, u_3, u_4\}$. We have $d_{D^\prime}^B(x) \geq {0}$, $d_{D^\prime}^B(y) \geq {1}$ and $d_{D^\prime}^B(z) \geq {1}$. With no loss of generality, suppose that $u_1\in N_{D^\prime}^B(y)$ and $u_2 \in N_{D^\prime}^B(z)$. The vertex $u_1$ has at least one blue edge to $Z$ since $\omega(G^R)=2n-1$. 
\\
Suppose that $u_1$ has a blue edge to $X$.
 By the fact that $d_{G^B}(u_1)\geq{4n-2}$, the vertex $u_1$ has at least $2n-4=4n-2-(2n-1+3)$ blue edges to $X\cup Z$. Let $L_{u_1}=N_{X\cup Z}^B(u_1)$. Then $|L_{u_1}|\geq{2n-4}$ and $L_{u_1}$ induces a red clique, otherwise $\{u_1,y\}\cup L_{u_1}$ induces a blue $K_4$. Now, $X\cup L_{u_1}$ or $Z\cup L_{u_1}$ induces a red $F_n$ since $n\geq{4}$, a contradiction.
\\ 
Now, suppose that all edges between $u_1$ and $X$ are red. Let $X^\prime=X\cup \{u_1\}$. Clearly, $X^\prime$ is a red clique of order $2n-1$. Since $\omega(G^R)=2n-1$, $u_2$ has some blue edges to $X^\prime$ and to $Y$. The vertex $u_2$ has at least $2n-3=4n-2-(2n-1+2)$ blue edges to $X^\prime \cup Y$. Let $L_{u_2}=N_{X^\prime \cup Y}^B(u_2)$. Then $|L_{u_2}|\geq{2n-3}$ and $L_{u_2}$ induces a red clique, otherwise $\{z,u_2\}\cup L_{u_2}$ induces a blue $K_4$, which is a contradiction. But, $X^\prime \cup L_{u_2}$ or $Y\cup L_{u_2}$ induces a red $F_n$ since $n\geq{4}$, a contradiction. 
 }\end{proof}
\begin{definition}\label{d} Let $n\geq{2}$ be an integer and $r=r(F_n, K_4)=6n + 1$. Define a class $\mathfrak{g}$ of graphs as the family of $G_1$ and $G_2$. Also, every graph in $\mathfrak{g}$ shows a red/blue edge-coloring of $K_{r-1}$ as
\\$G_1: G^{R}_1=3K_{2n} \qquad\qquad\, ,\qquad\qquad\qquad\,\,\,\, G^{B}_1=K_{2n,2n,2n}$
\\\\$G_2: G^{R}_2=3K_{2n}\cup (I_1\cup I_2\cup I_3) \qquad,\qquad G^{B}_2=K_{2n,2n,2n}-(I_1\cup I_2 \cup I_3)$
\\\\ such that if $K_{2n,2n,2n}=(A_1,A_2,A_3)$, then $I_i$  is a collection of $k_i$ independent edges between $A_i$ and $A_{i+1}$, for $i=1,2$ and $I_3$ is a collection of $k_3$ independent edges between $A_3$ and $A_1$, where $1\leq{k_i}\leq{2n}$, for $i=1,2,3$ and $G_2[I_1\cup I_2\cup I_3]$ contains no red triangle.
\end{definition}

\textbf{Proof of the Theorem \ref{t1}:} 
 Let $G=K_{6n}$ be a $(F_n,K_4)$-free graph. By Lemma \ref{l4}, $G$ contains a red $K_{2n}$. Thus, by Lemma \ref{l2}, $G$ contains a red $3K_{2n}$. 
One can easily find out that these two kind of colorings in the Definition \ref{d} are the only $(F_n,K_4)$-free colorings when $G$ contains a red $3K_{2n}$. \\ 
 We denote these red $K_{2n}$s by $G_i$, for $i=1,2,3$. Let $w$ be a new vertex and $G^\prime$ be a graph obtaining from $G$ by adding $w$ and some colored edges between $w$ and $G$. We want to find an upper bound for the most number of colored edges between $w$ and $G$  such that  $G^\prime$ still is a $(F_n,K_4)$-free graph. It is clear that $w$ has at most one red edge to $G_i$, for $i=1,2,3$, since $G^\prime$ has no red $F_n$.
 
 Suppose that $d_{G_i}^B(w)\geq 1$, $d_{G_j}^B(w)\geq 2$ and $d_{G_k}^B(w)\geq 3$, where $\{i,j,k\}=\{1,2,3\}$. With no loss of generality, let  $\{u_1\}\subseteq N_{G_i}^B(w)$, $\{u_2,u_3\}\subseteq  N_{G_j}^B(w)$ and $\{u_4,u_5,u_6\}\subseteq  N_{G_k}^B(w)$. Note that $d_{G_j}^R(u_1)\leq 1$ since $G$ contains no red $F_n$. So, we may assume that $u_1u_2$ is a blue edge. With the similar argument and with no loss of generality, suppose that $u_2u_4$ and $u_2u_5$ are blue edges. Now, $\{u_1,u_2,u_4,u_5\}$ contains a blue triangle.
 Hence $G^\prime$ contains a blue $K_4$, a contradiction.  
 
Now, suppose that $d_{G_i}^B(w)\geq 2$, for $i=1,2,3$.  With no loss of generality, let  $\{u_1,u_2\}\subseteq N_{G_1}^B(w)$, $\{u_3,u_4\}\subseteq  N_{G_2}^B(w)$ and $\{u_5,u_6\}\subseteq  N_{G_3}^B(w)$. One of the edges $u_1u_5$ or $u_1u_6$ is blue, say $u_1u_5$. If $u_1u_3$ and $u_1u_4$ are both blue, then $u_3u_5$ and $u_4u_5$ are red, otherwise $G^\prime$ contains a blue $K_4$. But now, $G$ contains a red $F_n$ with the center $u_5$, a contradiction. So, we may assume that $u_1u_3$, $u_2u_4$, $u_3u_5$ and $u_4u_6$ are red. Since $G$ contains no red $F_n$, $u_1u_5$ is blue. Note that $u_1u_4$ and $u_4u_5$ are blue. But now, $\{u_1,u_4,u_5,w\}$ induces a blue $K_4$, which is a contradiction.  

Now, we consider two cases:
\vspace{.6em}
\\
 \textbf{Case1:} $d_{G_i}^B(w)\geq 1$, for $i=1,2,3$.
 \vspace{.4em}
 \\
 With no loss of generality, according to what is mentioned above, assume that $d_{G_1}^B(w)=1$. If $d_{G_2}^B(w)=2$, then $d_{G_3}^B(w)\leq 2$ and hence $d(w)\leq 8$ (Since $d^R(w)\leq 3$). Now, if $d_{G_i}^B(w)=1$, for $i\in\{2,3\}$, then $d(w)\leq 2n+4$ (Since $d_{G_1\cup G_2}^R(w)\leq 2$).
 Hence in this case we have $d(w)\leq 2n+4$.
\vspace{.6em}
\\
 \textbf{Case2:} $d_{G_i}^B(w)=0$, for some $i\in\{1,2,3\}$. 
 \vspace{.4em}
 \\
 Let $i=1$. Since $d_{G_1}^R(w)\leq 1$, one can easily obtain that  $d(w)\leq 4n+1$ (at most $4n$ edges to $G_2\cup G_3$ and at most one red edge to $G_1$).
\vspace{.6em}
\\
These two cases imply that if $d(w)\geq{4n+2}$, then $G^\prime$ contains a red $F_n$ or a blue $K_4$.
Hence $r_*(F_n,K_4)\leq{4n+2}$.

 One can easily find out that if $w$ has degree $4n+1$ such that $4n$ blue edges are between $w$ and $G_1\cup G_2$ and one red edge is between $w$ and $G_3$, then $G^\prime$ is a $(F_n,K_4)$-free graph. Therefore,  $r_*(F_n,K_4)\geq{4n+2}$, and consequently $r_*(F_n,K_4)=4n+2$.  



\bibliographystyle{plain} 
\bibliography{mybib}{}

\end{document}